\documentclass{scrartcl}

\usepackage[T1]{fontenc}
\usepackage[utf8]{inputenc}

\usepackage{hyperref}
\hypersetup{
  colorlinks=true,
  linkcolor=blue,
  citecolor=blue,
  urlcolor=blue,
  pdftitle={Finite-time blow-up in the three-dimensional fully parabolic attraction-dominated attraction-repulsion chemotaxis system},
  pdfauthor={J. Lankeit},
}
\usepackage{amsmath,amsthm,amssymb}

\usepackage{xcolor}

\usepackage{newunicodechar}
\newunicodechar{κ}{\kappa}
\newunicodechar{ℝ}{\mathbb{R}}
\newunicodechar{χ}{\chi}
\newunicodechar{ξ}{\xi}
\newunicodechar{α}{\alpha}
\newunicodechar{β}{\beta}
\newunicodechar{γ}{\gamma}
\newunicodechar{ν}{\nu}
\newunicodechar{δ}{\delta}
\newunicodechar{Δ}{\Delta}
\newunicodechar{∇}{\nabla}
\newunicodechar{∂}{\partial}
\newunicodechar{σ}{\sigma}
\newunicodechar{π}{\pi}
\newunicodechar{η}{\eta}
\newunicodechar{μ}{\mu}
\newunicodechar{θ}{\theta}
\newunicodechar{λ}{\lambda}
\newunicodechar{ℕ}{\mathbb{N}}
\newunicodechar{∞}{\infty}
\newunicodechar{ρ}{\rho}
\newunicodechar{ε}{\varepsilon}

\newcommand{\f}[2]{\frac{#1}{#2}}
\newcommand{\nn}{\nonumber}
\newcommand{\ds}{\mathrm{d}s}

\newcommand{\dsigma}{\mathrm{d}\sigma}
\newcommand{\drho}{\mathrm{d}\rho}

\newcommand{\calF}{\mathcal{F}}

\newcommand{\calD}{\mathcal{D}}
\newcommand{\io}{\int_{\Omega}}
\newcommand{\kl}[1]{\left(#1\right)}
\newcommand{\ddt}{\frac{d}{dt}}
\newcommand{\wzero}{w^{(0)}}
\newcommand{\vzero}{v^{(0)}}
\newcommand{\Om}{\Omega}
\newcommand{\Ombar}{\overline{\Omega}}
\newcommand{\norm}[2][]{\|#2\|_{#1}}

\newcommand{\setl}[1]{\left\{#1\right\}}
\newcommand{\Lom}[1]{L^{#1}(\Omega)}
\newcommand{\Wom}[2]{W^{#1,#2}(\Omega)}
\newcommand{\Tmax}{T_{max}}
\newcommand{\embeddedinto}{\hookrightarrow}
\newcommand{\wtilde}{\widetilde{w}}
\newcommand{\ztilde}{\widetilde{z}}
\newcommand{\bdry}{\big|_{\partial\Omega}}

\newcommand{\init}{^{(0)}}

\DeclareMathOperator{\sgn}{sgn}

\newtheorem{lemma}{Lemma}[section]
\newtheorem{theorem}[lemma]{Theorem}
\newtheorem{remark}[lemma]{Remark}
\newtheorem{cor}[lemma]{Corollary}

\title{Finite-time blow-up in the three-dimensional fully parabolic attraction-dominated attraction-repulsion chemotaxis system}

\author{Johannes Lankeit\footnote{lankeit@ifam.uni-hannover.de}\\
{\small Leibniz Universität Hannover, Institut für Angewandte Mathematik}\\
{\small Welfengarten~1, 30167 Hannover, Germany}}

\begin{document}
\maketitle 
\begin{abstract}
 We show that the attraction-repulsion chemotaxis system  
\begin{equation*}
\begin{cases} u_t = \Delta u - \chi\nabla\cdot(u\nabla v_1) + \xi\nabla\cdot(u\nabla v_2)\\
 \partial_t v_1 = \Delta v_1 - \beta v_1 + \alpha u \\
 \partial_t v_2 = \Delta v_2 - \delta v_2 + \gamma u,
\end{cases}
\end{equation*}
posed with homogeneous Neumann boundary conditions in bounded domains $\Omega=B_R \subset \mathbb{R}^3$, $R>0$, admits radially symmetric solutions which blow-up in finite time if it is attraction-dominated in the sense that $\chi\alpha-\xi\gamma>0$.\\

\textbf{Math Subject Classification (MSC2020):} 35B44, 92C17; 35Q92, 35K55
\end{abstract}

\section{Introduction}
Chemotaxis, mathematically often captured in the Keller--Segel system (\cite{KS,Horstmann,BBTW})
\begin{equation}\label{KS}
 \begin{cases}
  u_t = Δu - χ∇\cdot(u∇v),\\
  v_t = Δv - v + u,
 \end{cases}
\end{equation}
is the directed motion of biological agents (usually cells) towards ($χ>0$, chemo-attraction) or away from ($χ<0$, chemo-repulsion) higher concentrations of a signal substance. 

At least in the chemo-attractive case, in \eqref{KS} aggregation can be observed in the extreme form of finite-time blow-up (for solutions arising from some initial data), that is 
\[
 \limsup_{t\nearrow T} \norm[\Lom\infty]{u(\cdot,t)}=\infty
\]
with some finite $T>0$, (see e.g. \cite{win_bu_higherdim} for the higher- and \cite{mizo_win,horstmann_wang} for the 2-dimensional case, as well as \cite{nagai2001,herrero_velazquez,jaeger_luckhaus} for closely related systems; or the survey \cite{lan_win_survey}). Chemotaxis appearing in form of repulsion instead, on the other hand, can facilitate proofs of classical solvability or long-time behaviour (see, e.g. \cite{tomek_philippe_cristian}, \cite{frederic_init}).

These antithetical outcomes prompt the question: What happens if \textit{both} effects are present in the same model? 

Indeed, such a system, namely 
\begin{equation}\label{attraction-repulsion}
\begin{cases} u_t = Δu - χ∇\cdot(u∇v_1) + ξ∇\cdot(u∇v_2) \qquad  &\text{ in }\Omega\times (0,T),\\
 ∂_t v_1 = Δv_1 - βv_1 +αu &\text{ in }\Omega\times(0,T),\\
 ∂_t v_2 = Δv_2 - δv_2 +γu&\text{ in } \Omega\times(0,T)\\
∂_{ν} u = ∂_{ν} v = ∂_{ν} w = 0 &\text{ on } ∂\Omega\times(0,T),\\
u(\cdot,t)=u^{(0)},\quad v_1(\cdot,0)=v^{(0)}, \quad v_2(\cdot,0)=v_2^{(0)}&\text{ in } \Omega.
\end{cases}
\end{equation}
has been suggested as model for (a possible explanation of) the formation of plaques during early stages of Alzheimer's disease  in \cite{luca2003chemotactic}, where microglia (with density $u$) react to different substances by being attracted (concentration $v_1$) or repelled ($v_2$) by them; both signals are produced by glial cells themselves (or by other cells, which are not part of the model, in response to their presence). 

All parameters in this system (describing the different strengths of chemo-attraction and -repulsion, of signal production and decay of the signals) are assumed to be positive. A first summary of the character of \eqref{attraction-repulsion} is given by the parameter combination $χα-ξγ$. If $χα-ξγ>0$, the system is ``dominated by attraction'', if $χα-ξγ<0$, it is ``repulsion-dominated''. A justification of this nomenclature and condition is provided by the results in \cite{tao_wang}, where it was shown for a parabolic-elliptic-elliptic simplification of \eqref{attraction-repulsion} (i.e. $∂_tv_1$ and $∂_tv_2$ in \eqref{attraction-repulsion} replaced by $0$) that for any dimension $n\ge 2$ global classical bounded solutions exist in the repulsion-dominated case, and under the condition $β=δ$ blow-up is possible in the attraction dominated case already in spatially two-dimensional domains for certain initial data. The requirement $β=δ$ was removed in the radial case in \cite{espejo_suzuki} and for nonradial settings in \cite{li_li,yu_guo_zheng}. As in \eqref{KS} with $χ>0$, solutions with critical mass (\cite{nagai_yamada_globalexistence}; for the determination of the critical value see \cite{guo_jiang_zheng}) exist globally even in the attraction-dominant 2D case; if their mass is subcritical, they are bounded (\cite{nagai_yamada_boundedness}). If solutions blow up, some guaranteed time of existence can be estimated from properties of the initial datum \cite{viglialoro}.

A comparable dichotomy between attraction- and repulsion-dominated case is also known for the parabolic-parabolic-elliptic system variant, where it was shown in \cite{jin_wang_jde} for two-dimensional domains that $χα-ξγ\le 0$ ensures global existence, whereas in the case of $χα-ξγ>0$ the initial mass $\io u_0$ decides the possibility of blow-up. If solutions are already supposed to be bounded, mild conditions ensuring their convergence are available, \cite{lin_mu_zhou}.

For the fully parabolic system, \eqref{attraction-repulsion}, it is known that solutions in one-dimensional domains are global and converge, \cite{liu_wang,jin_wang_asymptotic}. In the two-dimensional repulsion-dominant case, solutions are global and bounded, \cite{dongmei_youshan,jin}, and, again for the repulsion-dominant case, global weak solutions were constructed if $\Om\subset ℝ^3$ in \cite{jin}. If $χα-ξγ=0$, \cite{lin_mu_wang} proved boundedness of solutions (and convergence for small initial mass $\int u\init$). If the repulsion dominance is even stronger, in the sense that the quotient $\f{ξγ}{χα}$ surpasses a certain value, \cite{jin_wang_global_stabilization}, even treating unequal diffusion rates for the chemicals, achieved a convergence result for the 2D case.

(For the fully parabolic or the simplified systems there are several further works showing boundedness owing to additional system components which have been identified as beneficial for boundedness in the chemotaxis literature, like decaying sensitivity functions (e.g. \cite{chiyo_mizukami_yokota,wu_wu}), nonlinear diffusion of porous-medium (e.g. \cite{zheng,li_wang}) or $p$-Laplacian type \cite{liyan}, or the addition of logistic decay terms \cite{wang_zhuang_zheng,liyan_wang,chiyo_mizukami_yokota}.) 

Concerning blow-up, however, the only results for the fully parabolic system \eqref{attraction-repulsion} seems to be the recent note \cite{yutaro_tomomi_21}, where finite-time blow-up was shown for radial solutions in $n\ge 3$ if $β=δ$.

In this article, we will show that blow-up can occur even without the restrictive condition $β=δ$, and does occur for initial data close to any prescribed initial condition, in the following sense:
\begin{theorem}\label{thm}
 Let $n = 3$, $R>0$ and $\Om=B_R\subset ℝ^n$. Let $α,β,γ,δ,χ,ξ\in(0,\infty)$ be such that $χα-ξγ>0$. 
 Let $u\init\in C(\Ombar)$, $v_1\init\in\Wom1{∞}$, $v_2\init\in \Wom1{∞}$ be radially symmetric functions which are positive in $\Ombar$ and satisfy $\io u\init=m_u$ for some $m_u>0$. Then for each $p\in(1,\f{2n}{n+2})$ there exist sequences $(u^{(k)})_{k\inℕ}\subset C(\Ombar)$ and $(v_1^{(k)})_{k\inℕ}, (v_2^{(k)})_{k\inℕ} \subset \Wom1{∞}$ of radially symmetric nonnegative functions satisfying $\io u^{(k)}=m_u$ for all $k\inℕ$ and 
 \[
  u^{(k)}\to u\init \text{ in } \Lom p,\quad v_1^{(k)} \to v_1\init \text{ and } v_2^{(k)} \to v_2\init \text{ in } \Lom{\f{2n}{n-2}}
 \]
 as $k\to \infty$ and such that for every $k\inℕ$ the solution of \eqref{attraction-repulsion} with initial data $(u^{(k)},v_1^{(k)},v_2^{(k)})$ blows up within finite time.
\end{theorem}

One reason why blow-up has been detected in the parabolic-elliptic-elliptic case more frequently is that there the method of blow-up detection by a study of second moments can be employed, which for the parabolic-elliptic variant of \eqref{KS} goes back to \cite{nagai1995,biler_nadzieja_I,nagai2001} (for a short overview of the core idea see also \cite[Section 2.1.2]{lan_win_survey}). This method however, crucially relies on the equation for the signal being elliptic. 

Other blow-up proofs for \eqref{KS} stem from an energy functional, and, indeed, it is this functional on which \cite{yutaro_tomomi_21} (for $β=δ$) and the present article are based. In order to clarify the role of the condition $β=δ$  (and to better connect to results concerning \eqref{KS}), let us employ the following transformations in \eqref{attraction-repulsion}: 

We set 
\begin{subequations}\label{trafo}
\begin{align}
 w & = (χα-ξγ) u\label{trafo:w}\\
 z &= χv_1 - ξv_2\label{trafo:z}\\
 v &= v_1\label{trafo:v}\\[-5mm]
\intertext{and }
\label{paramtrafo}
 a=δ,\quad& b=(δ-β)χ,\quad c= β,\quad d=\f{α}{χα-ξγ}.
\end{align}
\end{subequations}
From now on, we always assume that $a>0$, $b\in ℝ$, $c>0$, $d>0$ (or, in terms of \eqref{attraction-repulsion}, that  $α,β,γ,δ,χ,ξ>0$ with $χα-ξγ>0$) and the domain $\Omega=B_R(0)\subset ℝ^n$ are fixed. (Accordingly, all constants appearing in the calculation may depend on these parameters, without this dependence being mentioned explicitly.) 

We then obtain that if $(u,v_1,v_2)$ is a solution of \eqref{attraction-repulsion}, then $(w,z,v)$ satisfies 
\begin{align*}
 w_t 
 & = (χα-ξγ) \kl{Δu - ∇\cdot(u ∇(-χv_1-ξv_2))}\\
 &= Δw - ∇\cdot(w∇z)\\[-8mm]
\end{align*}
and 
\begin{align}
 z_t 
  &= Δ(χv_1-ξv_2) - χδv_1 + χ(δ-β)v_1 + ξδv_2 + (χα-ξγ)u\label{zeqtrafo}\\
  &= Δz - az + b v + w \nn\\[-8mm]\nn
\end{align}
as well as 
\begin{align*}
 v_t &= 
 Δv_1 - βv_1 +αu\\
 &= Δv - cv + dw
\end{align*}
and vice versa.

If we additionally transform the inital data according to 
\begin{equation}\label{trafo:init}
  w\init = (χα-ξγ)u\init,\quad z\init=χv_1\init-ξv_2\init,\qquad v\init=v_1\init,
\end{equation}
we thus are interested in solutions of 
%
%
%
%
%
%
%
\begin{subequations}\label{attraction-repulsion-transformed}
\begin{align}
 w_t &= Δw - ∇\cdot(w∇z) \label{weq}\\
 z_t &= Δz - az + bv + w \label{zeq}\\
 v_t &= Δv - cv + dw \label{veq}\\
 ∂_{ν}w\bdry =& ∂_{ν} z\bdry = ∂_{ν} v\bdry = 0\label{trafo.bc}\\
 w(\cdot,0)=&w\init,\; z(\cdot,0)=z\init,\;v(\cdot,0)=v\init
\end{align}
\end{subequations}
In the attraction-dominant case -- which we are dealing with --, $χα>ξγ$ and hence $w\ge 0$. If $χα=ξγ$, the transformation \eqref{trafo:w} is inadvisable ($w=u$ being a better choice then), but the last term in the first line of \eqref{zeqtrafo} disappears and thus shows why the balanced case $χα=ξγ$ allows for better existence and boundedness results in \cite{lin_mu_wang} in the three-dimensional case. If $χα<ξγ$, similar transformations would still be possible, but the sign change of the source term $+dw$ in \eqref{veq} 
would considerably change the character of the system.

If $b=0$, that is, $β=δ$, the first two equations in \eqref{attraction-repulsion-transformed} are decoupled from the third and, more importantly, form the classical (attraction-only) Keller--Segel system of chemotaxis, \eqref{KS}, whose well-known blow-up results (here: \cite{win_bu_higherdim}) thereby can directly be transferred. (This is the observation \cite{yutaro_tomomi_21} is based on.) 

These blow-up proofs rest on the functional 
\begin{equation}\label{def:F}
 \calF (w,z) = \io w\ln w + \f a2 \io z^2 + \f12 \io |∇z|^2 - \io wz
\end{equation}
(if we assume $χ=1$ in \eqref{KS} and $a=1$). 
With 
\begin{equation}\label{def:D}
 \calD (w,z) = \io w|∇(\ln w -z)|^2 + \io \kl{az-Δz+w}^2
\end{equation}
we then have (for solutions $(w,z)$ of \eqref{KS}) that 
\begin{equation}\label{lf}
 \ddt \calF(w,z) + \calD(w,z) = 0.
\end{equation}
Therefore $\calF$ decreases along the trajectories of solutions. This Lyapunov functional is not only helpful in obtaining boundedness of solutions with small mass (on the set of which $\calF$ can be shown to be bounded from below in in two-dimensional scenarios, where also the resulting bounds can be used to bootstrap higher regularity, cf. \cite[Lemma 3.3]{BBTW}), but also extremely advantageous for finding blow-up solutions (see \cite[Section 2.1.3]{lan_win_survey}): It is known that $\calF(\wtilde,\ztilde)\ge -K$ for some $K>0$ on the set of stationary solutions $(\wtilde,\ztilde)$ and that every bounded solution must converge to such a stationary state (at least along a sequence of times $(t_j)_{j\inℕ}$ with $t_j\to \infty$). But since $\calF$ is decreasing and one can find initial data $(w\init,z\init)$ such that already $\calF(w\init,z\init)<-K$, the corresponding solutions cannot be global and bounded. 

However, despite \eqref{weq} coinciding with the first equation of \eqref{KS}, \eqref{attraction-repulsion-transformed} and \eqref{KS} are not identical. Let us compute $\ddt\calF(w,z)$ for solutions of \eqref{attraction-repulsion-transformed}: 
\begin{align*}
 \ddt \calF(w,z) &= \ddt \kl{\io w\ln w + \f a2 \io z^2 + \f12 \io |∇z|^2 - \io wz}\\
 &= \io w_t \ln w - \io w_t z + a\io z_tz + \io ∇z\cdot ∇z_t  - \io wz_t\\
 &= \io ∇\cdot\kl{w\kl{∇\ln w-∇z}} \kl{\ln w-z} + \io z_t\kl{az-Δz -w }\\
 &= -\io w|∇(\ln w - z)|^2 - \io \kl{az-Δz -w }^2 + b \io v\kl{az-Δz -w }.
\end{align*}
Apparently, we have 
\begin{equation}\label{energyinequality}
 \ddt \calF(w,z) + \calD(w,z) = b\io  v (az-Δz+w), 
\end{equation}
with no discernible sign on the right-hand side, thus losing $\calF$'s Lyapunov property.

Two remarks seem in order: Firstly, the related functional 
\[
 F(u,v_1,v_2)=\io u\ln u - χ\io uv_1 + ξ \io uv_2 + \f{χ}{2α}\io (βv_1^2+|∇v_1|^2) - \f{ξ}{2γ}\io (δv_2^2+|∇v_2|^2)
\]
has been identified as Lyapunov functional for the parabolic-parabolic-elliptic system in \cite{jin_wang_jde}. (We give it in the original variables, since \eqref{trafo:z}, or rather \eqref{zeqtrafo}, does not cope well with the equation for $v_2$ in \eqref{attraction-repulsion} being elliptic); however, for the computations (in particular identities like \cite[(5.65)]{jin_wang_jde}; see also \cite[(2.14)]{lin_mu_gao} for a related functional in a system with nonlinear diffusion), which ensured that no term like that on the right of \eqref{energyinequality} remained, the ellipticity of the equation for $v_2$ was imperative. 

Secondly: Actually, it is not important that $\calF$ be a Lyapunov functional. Already refining the above-sketched argument why a solution cannot be global and bounded to the assertion of blow-up within finite time (see \cite{win_bu_higherdim}) uses that $\calD$ can be estimated by a superlinear power of $\calF$, i.e. 
\begin{equation}\label{superlinearestimate}
\calD(w,z)\ge(-c\calF(w,z)-1)_+^{λ}
\end{equation}
for some $c>0, λ>1$ and all functions $(w,z)$ from a suitable set to which solutions belong, 
so that \eqref{lf} turns into 
\[
 \ddt \calF(w,z) + (-c\calF(w,z)-1)_+^{λ} \le 0
\]
or rather 
\begin{equation}\label{eq:busource}
 \ddt \kl{-c\calF(w,z)-1} \ge (-c\calF(w,z)-1)_+^{λ}, 
\end{equation}
ensuring finite-time blow-up whenever $-c\calF(w\init,z\init)-1>0$ due to $λ>1$. However, this line of reasoning can prevail over the substraction of an additional constant on the right of \eqref{eq:busource}. Hence, if we estimate the term in \eqref{energyinequality} by Young's inequality and the definition of $\calD$
\begin{equation}\label{estimatebadterm}
 b\io  v (az-Δz+w) \le \f12 \io \kl{az-Δz -w }^2 + \f{b^2}2 \io v^2 \le \f12 \calD(w,z) + \f{b^2}2 \io v^2,
\end{equation}
investing half of the dissipation term, we only need to ensure boundedness of $\norm[\Lom2]{v(\cdot,t)}$ to employ essentially the same argument. Said boundedness will be achieved in Section~\ref{sec:vbound} from a short application of semigroup estimates. (This is the only place where the spatial dimension may not exceed $3$. We will keep all other lemmata general, writing $n$ there.) Beforehand, we state a local existence result (Section~\ref{sec:locex}), prepare a superlinear estimate akin to \eqref{superlinearestimate} (Section~\ref{sec:superlinearestimate}) for functions belonging to a certain set $S(m,M,B,κ)$ (as defined in \eqref{defS}) and ensure that $(w(\cdot,t),z(\cdot,t))$ lie in this set if the initial data satisfy simple conditions (Lemma~\ref{lem:remaininset}).  In Section~\ref{sec:diffineq}, we combine all previous parts to obtain a differential inequality like \eqref{eq:busource}. After recalling unboundedness of $\calF$ in Section~\ref{sec:initdata}, we finally prove Theorem~\ref{thm} in Section~\ref{sec:proof}. 

\section{Local existence}\label{sec:locex}
In a first step we ensure existence of solutions. For the initial data we assume nonnegativity for two of the components and some regularity, 
\begin{equation}\label{init-reg}
 0\le w\init\in C(\Ombar), \qquad 0\le v\init \in C(\Ombar), \qquad z\init\in \Wom1{∞},
\end{equation}
with radial symmetry being a frequent later additional assumption.
\begin{lemma}\label{lem:ex}
Let $(w\init,z\init,v\init)$ be as in \eqref{init-reg}. Then there are $\Tmax>0$ and 
 \[
  (w,z,v)\in C^{2,1}(\Ombar\times(0,\Tmax))\cap C(\Ombar\times[0,\Tmax))
 \]
 such that $(w,z,v)$ solves \eqref{attraction-repulsion-transformed} classically and 
 \[
  \text{either } \Tmax=\infty \quad \text{ or } \quad \limsup_{t\nearrow \Tmax}\norm[\Lom\infty]{w(\cdot,t)}=\infty.
 \]
 Moreover, $(w,z,v) (\cdot,t)$ is radially symmetric for every $t\in(0,\Tmax)$ if $w\init,v\init,z\init$ are radially symmetric. Finally, $w\ge 0$ and $v\ge 0$ in $\Om\times(0,\Tmax)$.
\end{lemma}
\begin{proof}
This local existence result can be obtained by the fixed-point based reasoning well-established in the context of chemotaxis models, \cite[Lemma 3.1]{BBTW} (see also \cite[Lemma 3.1]{tao_wang}); nonnegativity for the first and third component follow from the comparison principle.
\end{proof}
\begin{remark}
 We do not assert positivity of the second component, since we do not want to impose a restriction on the sign of $b$.
\end{remark}

\section{Estimating \texorpdfstring{$\calF$}{F} in terms of \texorpdfstring{$\calD$}{D} on the set \texorpdfstring{$S$}{S}}\label{sec:superlinearestimate}
We let 
\begin{align}\label{defS}
 S(m,M,B,κ) =& \bigg\{(w,z)\in C^1(\Ombar)\times C^2(\Ombar)\mid w \text{ and } z \text{ are radially symmetric},\nn\\
 &\quad\qquad  w \text{ is positive }, \io w=m, \io |z| \le M, ∂_{ν}z=0 \text{ on } ∂\Omega, \text{ and } |z(x)|\le B|x|^{-κ}\bigg\}
\end{align}
and show that for $(w,z)\in S(m,M,B,κ)$, $\calD(w,z)$ corresponds to a superlinear power of $\calF(w,z)$. The key for this estimate is the following lemma:

\begin{lemma}\label{lem:estimate:wz}
 Let $m>0$, $M>0$, $κ>n-2$,  $B>0$. Then there are $C=C(m,M,B,κ)>0$ and $θ\in(\f12,1)$ such that for every $(w,z)\in S(m,M,B,κ)$ we have 
 \[
  \io w |z| \le C\cdot\kl{\norm[\Lom2]{Δz-az+w}^{2θ} + \norm[\Lom2]{\sqrt{w}\kl{∇\ln w - z}}+1}.
 \]
\end{lemma}
\begin{proof} For the case of $a=1$ and the additional assumption $z\ge 0$, this is the statement of \cite[Lemma 4.1]{win_bu_higherdim}. Only minor changes are required to make it applicable here. We include a brief discussion of the necessary adjustments in Appendix~\ref{Michaelsproof}.
\end{proof}

\begin{lemma}\label{lem:estimateforF}
 Let $m>0$, $M>0$, $κ>n-2$,  $B>0$. Then there are $C=C(m,M,B,κ)>0$ and $θ\in(\f12,1)$ such that for every $(w,z)\in S(m,M,B,κ)$ we have 
 \[
  \calF(w,z)\ge - C(m,M,B,κ) \kl{\calD^{θ}(w,z)+1}.
 \]
\end{lemma}
\begin{proof}
 As in \cite[Lemma 5.1]{win_bu_higherdim}, this follows directly from Lemma~\ref{lem:estimate:wz}: With some $C_1>0$
, $C_2>0$ and $θ\in(\f12,1)$, 
 \begin{align*}
  \calF(w,z) &= \io w\ln w + \f{a}2\io z^2+\f12\io|∇z|^2-\io wz\\
  &\ge -\f{|\Om|}{e} - \io w\left|z\right|\\
  &\ge - C_1 (1+\norm[\Lom2]{Δz-az+w}^{2θ} + \norm[\Lom2]{\sqrt{w}\kl{∇\ln w - v}})\\
  &\ge - C_2 (1+ ( \norm[\Lom2]{Δz-az+w}^{2} + \norm[\Lom2]{\sqrt{w}\kl{∇\ln w - v}}^2)^{θ})\\
  &= -C_2 (1+\calD^{θ}(w,z))
 \end{align*}
 for every $(w,z)\in S(m,M,B,κ)$.
\end{proof}

\section{Ensuring that solutions remain in \texorpdfstring{$S$}{S}} \label{sec:remaininS}
The crucial estimate in Lemma~\ref{lem:estimateforF} hinges on the fact that $(w,z)\in S(m,M,B,κ)$. In this section we will take care that $(w(\cdot,t),z(\cdot,t))\in S(m,M,B,κ)$ for every $t\in(0,\Tmax)$. We will focus on the component $z$ with the largest part of the work directed to the pointwise estimate $|z(x)|\le B|x|^{-κ}$. 
\begin{lemma}\label{lem:firstzestimates}
There is $C_1>0$ and for every 
  $p\in(1,\f{n}{n-1})$ there is $C_2(p)>0$ such that for any choice of $w\init, v\init, z\init$ as in \eqref{init-reg}, 
  the solution $(w,z,v)$ of \eqref{attraction-repulsion-transformed} satisfies 
\begin{equation}\label{zl1}
 \norm[\Lom1]{z(\cdot,t)}\le C_1 \kl{\norm[\Lom1]{w\init}+\norm[\Lom1]{z\init}+\norm[\Lom1]{v\init}}
\end{equation}
and 
 \begin{equation}\label{nablaz}
  \norm[\Lom p]{∇z(\cdot,t)} \le C_2(p) \kl{\norm[\Lom2]{∇z\init}+\norm[\Lom1]{v\init}+\norm[\Lom1]{w\init}}
 \end{equation}
 for every $t\in(0,\Tmax)$.
\end{lemma}
\begin{proof}
 An integration of \eqref{veq} and \eqref{weq} shows that 
 \[
  \ddt \io v = -c\io v + d \io w = - c \io v + d\io w\init \qquad \text{ in } (0,\Tmax)
 \]
and thus 
\[
 \io v(\cdot,t) \le \max\setl{\io v\init,\; \f{d}{c} \io w\init} \qquad \text{ for every } t\in(0,\Tmax), 
\]
so that by nonnegativity of $v$ and, again, mass-conservation for $w$, 
\begin{equation}\label{bv plus w}
 \norm[\Lom1]{bv(\cdot,t)+w(\cdot,t)}\le |b| \norm[\Lom1]{v\init} + \kl{ 1 + \f{|b|d}c}\norm[\Lom1]{w\init}
\end{equation}
for all $t\in(0,\Tmax)$.

On account of the representation 
\begin{equation}\label{z-duhamel}
 z(\cdot,t)=e^{t(Δ-a)} z\init + \int_0^t e^{-(t-s)a} e^{(t-s)Δ}(bv(\cdot,s)+w(\cdot,s)) \ds, \quad t\in(0,\Tmax),
\end{equation}
and the maximum principle for the heat equation we may estimate 
\begin{align*}
\norm[\Lom1]{z(\cdot,t)} &\le e^{-at} \norm[\Lom1]{z\init} + \int_0^t e^{-a(t-s)} \norm[\Lom1]{bv(\cdot,s)+w(\cdot,s)}\ds \\
&\le \norm[\Lom1]{z\init} +\f1a \sup_{s\in(0,\Tmax)} \norm[\Lom1]{bv(\cdot,s)+w(\cdot,s)}
\end{align*}
for every $t\in(0,\Tmax)$, which by \eqref{bv plus w} entails \eqref{zl1}.

Again starting from \eqref{z-duhamel}, we employ well-known estimates for the Neumann heat semigroup (see \cite[Lemma 1.3]{win_aggregation_vs}) and obtain a constant $c_1>0$ such that 
\begin{align*}
 \norm[\Lom p]{∇z(\cdot,t)} & \le c_1 \norm[\Lom 2]{∇z\init} + c_1 \int_0^t (t-s)^{-\f12-\f n2(1-\f1p)}e^{-a(t-s)} \norm[\Lom 1]{bv(\cdot,s)+w(\cdot,s)} \ds 
\end{align*}
holds for every $t\in(0,\Tmax)$ and every solution of \eqref{attraction-repulsion-transformed}. As the condition on $p$ ensures that 
\[
 c_2:=\int_0^\infty σ^{-\f12-\f n2(1-\f1p)} e^{-aσ} \dsigma < \infty, 
\]
this together with \eqref{bv plus w} implies \eqref{nablaz} if we set $C(p)=c_1+c_1c_2 (|b|+1+\f{|b|d}{c})$.
\end{proof}

In the following lemma we use the radial symmetry of solutions, writing $z(r,t)$ with $r=|x|$ in place of $z(x,t)$. The proof follows \cite[Lemma 3.2]{win_bu_higherdim} closely (where, however, additionally nonnegativity of $z$ was used). 
\begin{lemma}\label{lem:pointwise}
 Let $p\in(1,\f{n}{n-1})$. Then there exists $C(p)>0$ such that for every radially symmetric $(w\init,z\init,v\init)$ as in \eqref{init-reg}, 
 the solution of \eqref{attraction-repulsion-transformed} satisfies
 \[
  |z(r,t)| \le C(p) \kl{\norm[\Lom1]{w\init} + \norm[\Lom1]{v\init} + \norm[\Lom1]{z\init} +\norm[\Lom2]{∇z\init}} r ^{-\f{n-p}p}
 \]
 for all $(r,t)\in (0,R)\times(0,\Tmax)$.
\end{lemma}
\begin{proof}
If we set $K=C_1\kl{\norm[\Lom1]{w\init}+\norm[\Lom1]{z\init}+\norm[\Lom1]{v\init}}$ with $C_1$ from Lemma~\ref{lem:firstzestimates}, we know from said lemma that 
\[
 \int_{B_R\setminus B_{\f R2}} |z(\cdot,t)| \le \io |z(\cdot,t)| \le K
\]
for every $t\in(0,\Tmax)$, and thus for every $t\in(0,\Tmax)$ can find $r_0(t)\in(\f{R}2,R)$ such that 
\[
 |z(r_0(t),t)|\le \f{K}{|B_R\setminus B_{\f R2}|}.
\]
In particular, by Hölder's inequality and Lemma~\ref{lem:firstzestimates} (with $C_2(p)$ taken from the latter),
\begin{align*}
 |z(r,t)|-\f{K}{|B_R\setminus B_{\f R2}|} R^{\f{n-p}p} r^{-\f{n-p}{p}}
&\le |z(r,t)| - |z(r_0(t),t)| \le |z(r,t)-z(r_0,t)| \\
 &\le \left\lvert \int_{r_0(t)}^{r} z_r (ρ,t)\drho\right\rvert\\
  &\le 
  \left\lvert 
  \int_{r_0(t)}^{r} ρ^{n-1} |z_r(ρ,t)|^p\drho 
  \right\rvert^{\f1{p}}
  \left\lvert \int_{r_0(t)}^{r} ρ^{-\f{n-1}{p-1}} \drho\right\rvert^{\f{p-1}p}\\
  &\le C_2(p)  \kl{\norm[\Lom2]{∇z\init}+\norm[\Lom1]{v\init}+\norm[\Lom1]{w\init}}
  \left\lvert \int_{r_0(t)}^{r} ρ^{-\f{n-1}{p-1}} \drho\right\rvert^{\f{p-1}p}
\end{align*}
for every $(r,t)\in(0,R)\times(0,\Tmax)$. Since 
\[
 \left\lvert \int_{r_0(t)}^{r} ρ^{-\f{n-1}{p-1}} \drho\right\rvert^{\f{p-1}p}\le 2^{\f{n-p}p}\kl{\f{p-1}{n-p}}^{\f{p-1}p}  r^{-\f{n-p}p}
\]
for every $r\in(0,R)$ and $t\in(0,\Tmax)$ (cf. \cite[(3.6), (3.7)]{win_bu_higherdim}) due to $r_0(t)>\f R2$, the lemma follows.
\end{proof}

In the same way as in \cite[Cor. 3.3]{win_bu_higherdim}, we summarize these results: 
\begin{cor}\label{cor z}
For every $κ>n-2$ there is $C(κ)>0$ such that for all radially symmetric $(w\init,z\init,v\init)$ as in \eqref{init-reg}, the solution of \eqref{attraction-repulsion-transformed} satisfies
 \[
  |z(r,t)| \le C(κ) \kl{\norm[\Lom1]{w\init} + \norm[\Lom1]{v\init} + \norm[\Lom1]{z\init} +\norm[\Lom2]{∇z\init}} r^{-κ}
 \]
 for all $(r,t)\in (0,R)\times(0,\Tmax)$.
\end{cor}
\begin{proof}
 This follows upon an application of Lemma~\ref{lem:pointwise} to some $p>1$ fulfilling $p\in(\f{n}{κ+1},\f{n}{n-1})$.
\end{proof}

\begin{lemma}\label{lem:remaininset}
 Let $m>0$ and $A>0$ and $κ>n-2$. Then there are $M>0$ and $B>0$ such that for every choice of radially symmetric $(w\init,z\init,v\init)$ as in \eqref{init-reg} and with 
 \[
  \io w\init = m, \qquad \norm[\Lom1]{v\init}\le A, \qquad \norm[\Lom1]{z\init}\le A, \qquad \norm[\Lom2]{∇z\init}\le A
 \]
 the solution $(w,z,v)$ of \eqref{attraction-repulsion-transformed} satisfies 
 \[
  (w(\cdot,t),z(\cdot,t))\in S(m,M,B,κ) \qquad \text{ for every } t\in(0,\Tmax).
 \]

\end{lemma}
\begin{proof}
 This is a consequence of mass-conservation of $w$, \eqref{zl1} (indicating a suitable choice of $M$) and Corollary~\ref{cor z} (yielding $B$).
\end{proof}

\section{The time-uniform bound for \texorpdfstring{$v$}{v} in \texorpdfstring{$L^2(\Omega)$}{L²(Ω)}} \label{sec:vbound}
As discussed in the introduction, here we give an estimate for the seemingly inconvenient term in \eqref{estimatebadterm}:
\begin{lemma}\label{lem:vbounded-l2}
 Let $n=3$, $A>0$ and $m>0$. Then there is $C=C(A,m)>0$ such that whenever $\wzero$ satisfies \eqref{init-reg} and $\io \wzero\le m$ and $\vzero$ is such that $\norm[\Lom2]{\vzero}\le A$, then any solution of \eqref{attraction-repulsion-transformed} satisfies 
 \begin{equation}\label{vbounded-l2}
  \norm[\Lom2]{v(\cdot,t)} \le C \qquad \text{for all } t\in(0,\Tmax).
 \end{equation}
\end{lemma}
\begin{proof}
 Representing $v$ by means of the Neumann heat semigroup, we find that according to \cite[Lemma 1.3]{win_aggregation_vs} there are $c_1, c_2>0$ such that for every $t\in(0,\Tmax)$, 
 \begin{align*}
  \norm[\Lom2]{v(\cdot,t)} & = \norm[\Lom2]{e^{t(Δ-c)} \vzero + \int_0^t e^{(t-s)(Δ-c)} w(\cdot,s) \ds}\\
  &\le c_1 e^{-ct} \norm[\Lom2]{\vzero} + c_2\int_0^t e^{-c(t-s)} (1+(t-s))^{-\f32(1-\f12)} \norm[\Lom1]{w(\cdot,s)} \ds.
 \end{align*}
 Since $\int_0^\infty e^{-cσ}σ^{-\f34} \dsigma$ is finite and $\norm[\Lom1]{w(\cdot,s)}=\norm[\Lom1]{\wzero}$ for every $s\in(0,\Tmax)$, this shows \eqref{vbounded-l2}.
\end{proof}

\section{The differential inequality}\label{sec:diffineq}
This section is dedicated to the differential inequality, from which the blow-up arises. Let us first summarize the outcome of the corresponding discussion in the introduction: 

\begin{lemma}\label{lem:diffineq-conditional}
 Let $(w\init,z\init,v\init)$ be as in \eqref{init-reg}. Then every solution $(w,z,v)$ of \eqref{attraction-repulsion-transformed} fulfilling 
 \begin{equation}\label{assumedbound}
  \norm[\Lom2]{v(t)} \le K 
  \qquad \text{ for all } t\in (0,\Tmax)
 \end{equation}
 satisfies 
 \begin{equation}\label{ddtF}
 \ddt \calF(w,z) + \f12\calD(w,z) \le  \f{b^2K^2}{2} \qquad \text{ in } (0,\Tmax).
 \end{equation}
\end{lemma}
\begin{proof}
By \eqref{energyinequality} and \eqref{estimatebadterm} together with \eqref{assumedbound}, we directly obtain that 
 \begin{align*}
 \ddt \calF(w,z) + \calD(w,z) 
 &\le \f12 \calD(w,z) + \f{b^2K^2}{2}
 \end{align*}
 and hence \eqref{ddtF}.
\end{proof}

Thanks to Lemma~\ref{lem:vbounded-l2}, this statement can be transformed to pose conditions on the solution at the initial time only. If additionally combined with the outcomes of Section~\ref{sec:superlinearestimate} and Section~\ref{sec:remaininS}, we obtain the following: 

\begin{lemma}\label{lem:diffineq-final}
 Let $n=3$, $m>0$ and $A>0$. Then there are $θ\in(\f12,1)$, $C_1>0$, $C_2>0$ and $C_3>0$ such that for every choice of radially symmetric $(w\init,z\init,v\init)$ as in \eqref{init-reg} 
with 
 \begin{align*}
  \io w\init &= m, &\qquad \norm[\Lom1]{v\init}&\le A, &\qquad \norm[\Lom1]{z\init}&\le A, \qquad \\
  &&\norm[\Lom2]{v\init}&\le A,&  \norm[\Lom2]{∇z\init}&\le A
 \end{align*}
 the solution $(w,z,v)$ of \eqref{attraction-repulsion-transformed} satisfies 
 \begin{equation}\label{diffineq}
  \ddt \kl{-\f1{C_1}\calF(w,z)-1 } \ge C_2 \kl{-\f1{C_1}\calF(w,z)-1 }_+^{\f1{θ}}-C_3 \qquad \text{ in } (0,\Tmax).
 \end{equation}
\end{lemma}
\begin{proof}
 We let $κ>n-2$ and according to Lemma~\ref{lem:remaininset}, we can find $M$, $B$ such that every solution emanating from initial data as in the lemma satisfies  $(w(\cdot,t),z(\cdot,t))\in S(m,M,B,κ)$ for every $t\in(0,\Tmax)$. Thus Lemma~\ref{lem:estimateforF} asserts the existence of some constant $c_1>0$ such that with $θ\in(\f12,1)$ as in Lemma~\ref{lem:estimateforF}
 \[
  \calF(w,z)\ge -c_1(\calD^{θ}(w,z)+1),  
 \]
 and hence 
 \[
  \calD(w,z)\ge \kl{-\f1{c_1}\calF(w,z)-1}_+^{\f1{θ}}
  \qquad \text{ throughout } (0,\Tmax)
 \] 
 for every such solution.  As Lemma~\ref{lem:vbounded-l2} ensures that with some $c_2=c_2(A,m)>0$, 
 \[
  \norm[\Lom2]{v(\cdot,t)}\le c_2 \qquad \text{ for every } t\in (0,\Tmax),
 \]
 Lemma~\ref{lem:diffineq-conditional} becomes applicable and shows that with $c_3=\f{b^2c_2^2}{2}$
 \[
   \ddt \calF(w,z)  \le  c_3 - \f12\calD(w,z) \qquad \text{ in } (0,\Tmax)
 \]
 and therefore 
 \[
  \ddt \kl{-\f1{c_1}\calF(w,z)-1} \ge - \f{c_3}{c_1} + \f1{2} \kl{-\f1{c_1}\calF(w,z)-1}_+^{\f1{θ}} \qquad \text{ in } (0,\Tmax).
  \qedhere
 \]
\end{proof}
In terms of blow-up, we can conclude the following from Lemma~\ref{lem:diffineq-final}:
\begin{lemma}\label{lem:bu}
 Let $n=3$, $m>0$ and $A>0$. Then there is $K>0$ such that for every choice of radially symmetric $(w\init,z\init,v\init)$ as in \eqref{init-reg} 
with 
 \begin{align*}
  \io w\init = m, &\qquad \norm[\Lom1]{v\init}\le A, \qquad \norm[\Lom1]{z\init}\le A,\\& \qquad ,\norm[\Lom2]{v\init}\le A \qquad \norm[\Lom2]{∇z\init}\le A
 \end{align*}
 and 
 \[
  \calF(w\init,z\init)< - K
 \]
 the solution $(w,z,v)$ of \eqref{attraction-repulsion-transformed} blows up in finite time.
\end{lemma}
\begin{proof}
 Given $m$ and $A$, we let $C_1, C_2, C_3$ be as provided by Lemma~\ref{lem:diffineq-final}. Setting $K=C_1\kl{\kl{\f{C_3}{C_2}}^{θ}+1}$, we conclude from \eqref{diffineq} that $y(t)=-\f1{C_1}\calF(w(\cdot,t),z(\cdot,t))-1$ satisfies 
 the superlinear differential inequality \[y'\ge C_2 y_+^{\f1{θ}}-C_3\] together with $y(0)> \f1{C_1} K - 1$, which ensures that $C_2y_+^{\f1{θ}}(0)- C_3 >0$. Hence $y$ (therefore $\calF(w,z)$ and thus $w$) must blow up in finite time. 
\end{proof}

\section{Preparing initial data}\label{sec:initdata}

It is well-known that $\calF$ is unbounded from below (even on each set of pairs of functions with fixed $L^1$-norm of the first component). We recall the corresponding result: 
\begin{lemma}\label{lem:Funbounded}
 Let $n\ge3$, $m>0$ and $w\in C(\Ombar)$ and $z\in \Wom1{∞}$ be radially symmetric and let $w$ be positive with $\io w=m$. Then for each $p\in (1,\f{2n}{n+2})$ there are sequences $(w^{(k)})_{k\inℕ}\subset C(\Ombar)$ and $(z^{(k)})_{k\inℕ}\subset\Wom1{∞}$ of radially symmetric functions such that $w^{(k)}$ is positive and satisfies $\io w^{(k)}=m$ for every $k\inℕ$ and such that 
 \[
  w^{(k)}\to w \;\text{ in } \Lom p \quad \text{and}\quad z^{(k)}\to z \text{ in } \Wom12 \qquad\text{as } k\to\infty
 \]
 but such that with $\calF$ from \eqref{def:F} 
 \[
  \calF(w^{(k)},z^{(k)})\to -∞ \qquad \text{as } k\to\infty.
 \]
\end{lemma}
\begin{proof}
 This is can be shown by the construction in \cite[Lemma 6.1]{win_bu_higherdim} (see also \cite{lankeit_itbu}). 
\end{proof}

\section{Proof of Theorem~\ref{thm}}\label{sec:proof}
In order to prove Theorem~\ref{thm}, we mainly have to combine Lemma~\ref{lem:diffineq-final} with Lemma~\ref{lem:Funbounded}. Additionally, we return to the variables of the original system \eqref{attraction-repulsion}.

\begin{proof}[Proof of Theorem~\ref{thm}]
 In line with \eqref{trafo:init}, we define 
 \[
  w\init = (χα-ξγ)u\init,\quad z\init=χv_1\init-ξv_2\init,\qquad v\init=v_1\init.
 \]
We let $m=\io w\init$, 
\[
 A=2\max\setl{\norm[\Lom2]{v\init},\norm[\Lom1]{v\init},\norm[\Lom1]{z\init},\norm[\Lom2]{∇z\init}}
\]
and introduce $K>0$ as provided for these choices of $A$ and $m$ by Lemma~\ref{lem:bu}. Aided by Lemma~\ref{lem:Funbounded}, we then pick a sequence $((w^{(k)},z^{(k)}))_{k\inℕ}$ of radially symmetric functions satisfying $w^{(k)}\ge 0$ in $\Ombar$ and $\io w^{(k)}=m$ for every $k\inℕ$, 
\[
 w^{(k)}\to w\init \text{ in } \Lom p\quad \text{ and } z^{(k)}\to z\init \text{ in } \Wom12
\]
as well as $\calF(w^{(k)},z^{(k)})\to -\infty$ as $k\to \infty$. We then let $v^{(k)}=\f1{χ}(ξv_2\init+z^{(k)})_+$ and note that due to $\Wom12\embeddedinto \Lom{\f{2n}{n-2}}$ the term $\f1{χ}(ξv_2\init+z^{(k)})$ converges in $\Lom{\f{2n}{n-2}}$ and the same holds true for its positive part:
\[
 v^{(k)} \to \f1{χ}\kl{ξv_2\init+z\init}_+ = \f1{χ}\kl{χv_1\init}_+=v_1\init =v\init \quad \text{in } \Lom{\f{2n}{n-2}}\text{ as } k\to\infty.
\]
If necessary discarding a finite number of elements from the sequence, on account of the convergence we may assume that 
\[
 \norm[\Lom2]{v^{(k)}}\le A,\quad \norm[\Lom1]{v^{(k)}}\le A, \quad \norm[\Lom1]{z^{(k)}} \le A, \quad \norm[\Lom2]{∇z^{(k)}}\le A
\]
and 
\[
 \calF(w^{(k)},z^{(k)})\le -K
\]
for every $k\inℕ$. Lemma~\ref{lem:bu} then asserts that the corresponding solutions $(w,z,v)$ blow up in finite time. Defining $u=\f1{χα-ξγ}w$, $v_1=v$ and $v_2=\f1{ξ}(χv-z)$ we have blow up of solutions to the original system \eqref{attraction-repulsion}. Concerning nonnegativity of the corresponding initial data $u^{(k)}=\f1{χα-ξγ}w^{(k)}$, $v_1^{(k)}=v^{(k)}$ and $v_2^{(k)}=\f1{ξ}(χv^{(k)}-z^{(k)})$ let us note that nonnegativity of $u^{(k)}$ immediately follows from that of $w^{(k)}$, that of $v_1^{(k)}$ is ensured by the positive part in the definition of $v^{(k)}$ and for 
\[
ξv_2^{(k)}=χv^{(k)}-z^{(k)}=(ξv_2\init + z^{(k)})_+ - z^{(k)}\ge 0
\]
we can rely on the nonnegativity of $v_2\init$.
\end{proof}

\appendix 
\section{Proof of Lemma~\ref{lem:estimate:wz}}\label{Michaelsproof}
For Lemma~\ref{lem:estimate:wz}, we have referred to \cite{win_bu_higherdim}. However, there only the case of $a=1$ was covered and, more importantly, the definition of $S(m,M,B,κ)$ contained the condition that $z$ be nonnegative. In order to avoid a gap in the proof, we discuss necessary changes. However, since most necessary adjustments are rather small, we will confine ourselves to a brief outline. We already change the variable names $(u,v)$ to $(w,z)$ in line with the notation in the earlier proofs.\\ 

Throughout this section, we assume that $m>0$, $M>0$, $B>0$, $κ>n-2$, $n\ge3$, $a>0$ are fixed. \\

Following \cite[(4.5) and (4.6)]{win_bu_higherdim}, but including $a$, we introduce the abbreviations 
\[
 f=-Δz+az-w, \qquad g=\kl{\f{∇w}{\sqrt{w}}-\sqrt{w}∇z}\cdot\f{x}{|x|}=\sqrt{w}∇\kl{\ln w - z}\cdot\f{x}{|x|}. 
\]

\begin{lemma}\label{m4.2}
 There exists $C(M)>0$ such that for all $(w,z)\in S(m,M,B,κ)$ we have 
 \[
  \io w|z| \le 2 \io |∇z|^2 + C(M) \kl{\norm[\Lom2]{Δz-az+w}^{\f{2n+4}{n+4}}+1}.
 \]
\end{lemma}
\begin{proof}
 If we multiply the definition of $f$ by $|z|$, we obtain
 \[
  \io w|z| =\io \f{z}{|z|}|\nabla z|^2 + a\io z|z|- \io f|z|.
 \]
 Since $\left|\io \f{z}{|z|}|\nabla z|^2\right|\le \io |∇z|^2$, $\left|a\io z|z|\right|\le \f12\io |∇z|^2 + c_2$ for some suitable $c_2>0$ and $\left|-\io fz\right|\le \norm[\Lom2]{f}\norm[\Lom2]{z}$, the proof of \cite[Lemma 4.2]{win_bu_higherdim} remains applicable.
\end{proof}

\begin{lemma}\label{m4.3}
 Let $r_0\in(0,R)$ and $ε\in(0,1)$. Then there is $C(ε,m,M,B,κ)>0$ such that for all $(w,z)\in S(m,M,B,κ)$ we have 
 \[
  \int_{\Om\setminus B_{r_0}} |∇z|^2 \le ε\io w|z| + ε\io |∇z|^2 + C(ε,m,M,B,κ) \kl{r_0^{-\f{2n+4}n κ}+ \norm[\Lom2]{Δz-az+w}^{\f{2n+4}{n+4}}}.
 \]
\end{lemma}
\begin{proof}
 Testing the definition of $f$ by $\sgn(z)|z|^{α}$ for some $α\in(0,1)$, we obtain $α\io |z|^{α-1}|∇z|^2+ a\io |z|^{α+1} = \io w|z|^{α}+\io f|z|^{α}$. Merely replacing every other occurrence of $v$ or $v^{α}$ in the proof of \cite[Lemma 4.3]{win_bu_higherdim} by $|z|$ or $|z|^{α}$ instead of by $z$ and $z^{α}$, we can copy said proof.
\end{proof}

\begin{lemma}\label{m4.4}
 There is $C(m)>0$ such that for every $r_0\in(0,R)$ and all $(w,z)\in S(m,M,B,κ)$ we have 
 \[
  \int_{B_{r_0}} |∇z|^2 \le C(m) \kl{r_0\norm[\Lom2]{Δz-az+w}^2 + \norm[\Lom2]{\sqrt{w}∇(\ln w-z)} + \norm[\Lom2]{z}^2 + 1}.
 \]
\end{lemma}
\begin{proof}
This is \cite[Lemma~4.4]{win_bu_higherdim} and apart from an additional coefficient $a$ in the definition of $f$ in the first two lines of the proof and in front of the last term in \cite[(4.25)]{win_bu_higherdim} and the preceding inequality, no changes are necessary.
\end{proof}

\begin{lemma}\label{m4.5}
 For each $ε>0$ there is $C(ε,m,M,B,κ)>0$ such that every $(w,z)\in S(m,M,B,κ)$ obeys the estimate 
 \[
  \io |∇z|^2\le ε\io w|z| + C(ε,m,M,B,κ)\kl{\norm[\Lom2]{Δz-az+w}^{2θ}+\norm[\Lom2]{\sqrt{w}∇(\ln w-z)}+1},
 \]
where $θ=\f{1}{1+\f{n}{(2n+4)κ}}\in(\f12,1)$.
\end{lemma}
\begin{proof}
 Here the term $\io uv$ arising from an application of \cite[Lemma~4.3]{win_bu_higherdim} (i.e. Lemma~\ref{m4.3}) has to be replaced by $\io w|z|$. Apart from that, we can follow the proof in the preprint of \cite{win_bu_higherdim} on arxiv.org. (Note that the seemingly slightly better exponent in the published journal version stems from an unfortunate mistake near the end of the proof.)
\end{proof}

\begin{proof}[Proof of Lemma~\ref{lem:estimate:wz}]
 Lemma~\ref{lem:estimate:wz} follows from a combination of Lemma~\ref{m4.2} with Lemma~\ref{m4.5} applied to $ε=\f14$, since $\f{2n+4}{n+4}\le2θ$.
\end{proof}

{
\footnotesize
\setlength{\parskip}{0pt}
\setlength{\itemsep}{0pt plus 0.3ex}

}

\end{document}